\documentclass[a4paper,11pt]{amsart}

\usepackage[utf8]{inputenc}
\usepackage{amsmath,amsfonts,amscd,amssymb,amsthm,epsfig,euscript}
\usepackage{mathrsfs}
\usepackage{amsxtra}
\usepackage{hyperref}
\hypersetup{colorlinks,citecolor=blue,urlcolor=blue,plainpages=false,hypertexnames=false}

\newtheorem{thm}{Theorem}
\newtheorem{lem}{Lemma}
\newtheorem{pro}{Proposition}
\newtheorem{cor}{Corollary}
\theoremstyle{remark}

\numberwithin{equation}{section}
\usepackage[all]{xy}

\newcommand{\cal}{\mathcal}

\title{On orthogonal decompositions of hermitian Higgs bundles}

\author[Cardona]{Sergio A. H. Cardona}
\address{Secihti Research Fellow--Instituto de Matem\'aticas, Universidad Nacional \indent Aut\'onoma de 
M\'exico, Le\'on 2 altos, Col. centro, 68000, Oaxaca, Mexico}
\email{sholguin@im.unam.mx} 

\author[Mart\'inez-Ruiz]{Kenett Mart\'inez-Ruiz}
\address{Instituto de Matem\'aticas, Universidad Nacional Aut\'onoma de 
M\'exico, Le\'on \indent 2 altos, Col. centro, 68000, Oaxaca, Mexico}
\email{kenettmartinez@im.unam.mx} 


\begin{document}

\maketitle 

\begin{abstract}
A hermitian Higgs bundle is a triple $({\mathfrak E},h) = (E,\Phi, h)$, where ${\mathfrak E}=(E,\Phi)$ is a Higgs bundle and $(E,h)$ is a holomorphic hermitian vector bundle. It is well-known that several results on holomorphic vector bundles extend to the Higgs bundles setting, although this is not always the case. In this article we show that some classical propositions, involving orthogonal decompositions of holomorphic hermitian vector bundles and the second fundamental form of its holomorphic subbundles, can be extended to hermitian Higgs bundles. The extended propositions concerning orthogonal decompositions have immediate applications in Higgs bundles, and we mention some of these throughout the article. Moreover, the extended propositions concerning the second fundamental form are generalizations of previously known results on Higgs bundles. In particular, here we include alternative proofs of these extended propositions without using local computations. Finally, as an application of the above results and due to the lack of a certain parallelism condition, we show that a classical theorem concerning the Kobayashi functional for holomorphic vector bundles does not admit a straightforward extension to Higgs bundles.\\


\end{abstract}
\section{Introduction} \label{Pre-Intro.}

Roughly speaking a Higgs bundle is a holomorphic vector bundle, together with a section of a certain associated holomorphic vector bundle. Since its introduction by Hitchin \cite{Hitchin} and Simpson \cite{Simpson, Simpson 2}, these geometric objects have played an important role in complex geometry as well as in mathematical physics. For instance, from the viewpoint of geometry, several results on holomorphic bundles concerning special metrics and connections can be extended to Higgs bundles \cite{Bruzzo-Granha, Bruzzo-Granha 2, Cardona 1, Cardona 2, Cardona 7, Li, Seaman}. Similarly, from the viewpoint of physics, Higgs bundles are closely related to dimensional reduction. In fact, in the pioneering work \cite{Hitchin}, Higgs bundles arise as geometric objects associated to Hitchin's equations (a two dimensional reduction of the self-dual Yang-Mills equations in four dimensions). On the other hand, Hitchin's equations can also be seen as the lowest dimensional case of another set of equations introduced by Ward \cite{Ward 2} and which are a dimensional reduction of Yang-Mills equations in higher dimensions. The equations introduced by Ward are known as $2k$-Hitchin's equations and can be directly related to Higgs bundles \cite{Cardona 9}. It is important to mention that from the very beginning, Higgs bundles are associated to another set of equations, usually known as the Hermitian-Yang-Mills equations \cite{Simpson}. Higgs bundles are to the Hermitian-Yang-Mills equations as holomorphic vector bundles are to the Hermitian-Einstein equations \cite{Kobayashi, Uhlenbeck-Yau}. Today, Higgs bundles and all the above equations are of major interest in geometry and physics and there exists a very ample literature related to the topic, see for instance \cite{Biswas-Schumacher, Kneipp, Ward, Wijnholt}. In particular, Higgs bundles are geometric objects of relevance in the foundational work of Kapustin and Witten \cite{Witten-Kapustin}. 

This article is organized as follows. In Section \ref{Intro.} we review part of the elementary theory on holomorphic bundles that will be important later on. In particular, in this section we explicitly recall and compile some classical results on holomorphic hermitian vector bundles contained within \cite{Kobayashi}, which are otherwise distributed throughout different chapters. In this section we also establish part of the notation and terminology that will be used throughout the article. In Section \ref{Higgs 2nd. fund.} we summarize some basic concepts and properties on Higgs bundles. We also mention two lemmas concerning short exact sequences and invariance conditions on Higgs bundles that will be relevant in the remaining part of the article. In this section we also revisit some propositions on Higgs bundles \cite{Bruzzo-Granha}, which are extensions of well-known results associated to the second fundamental form of holomorphic subbundles of holomorphic hermitian vector bundles \cite{Kobayashi}. We include here alternative proofs of these results that do not use local computations. In Section \ref{Herm. Higgs b.} we establish a couple of propositions concerning orthogonal decompositions of hermitian Higgs bundles. In particular, we include proofs of these results that rely heavily on classical results on holomorphic hermitian vector bundles as well as standard definitions on Higgs bundles. We then show immediate applications of these established results. More specifically, we show that by utilizing the aforementioned lemmas and by introducing the notion of $C^{\infty}$ Higgs orthogonal complement, the extended results on Higgs bundles involving orthogonal decompositions imply two corollaries, which can be established by referring exclusively to objects in the Higgs category. Finally, in Section \ref{Sec. cal J} we briefly review some properties of the Kobayashi functional which are included in \cite{Cardona 7} and we show that,  due to the lack of a parallelism condition of the Hitchin-Simpson connection, a classical theorem concerning the Kobayashi functional for holomorphic vector bundles does not admit a straightforward extension to Higgs bundles. 

The main purpose of this article is to establish a couple of propositions as well as to revisit two other results on hermitian Higgs bundles \cite{Bruzzo-Granha}. The first couple of propositions are natural extensions of well-known results on orthogonal decompositions of holomorphic hermitian vector bundles. These classical results are included in the form of propositions 4.18 and 6.14 in the first chapter of the celebrated textbook of Kobayashi \cite{Kobayashi}. There are two other classical results, closely related to the previous propositions and which appear in the same reference as propositions 6.4 and 6.6. Extensions of these latter results to hermitian Higgs bundles have been previously studied in \cite{Bruzzo-Granha}. However, as we will see, they are are only special cases of more general propositions. For the reader's convenience, in the following section we review some elementary definitions, as well as summarize the classical results on holomorphic vector bundles that are strictly necessary for understanding the article (for more details the reader can see \cite{Demailly, Griffiths-Harris, Kobayashi}).
\section{Preliminaries}\label{Intro.}

Throughout this article, $M$ is an $n$-dimensional compact K\"ahler manifold with K\"ahler form $\omega$ and $E$ is a rank $r$ holomorphic vector bundle over it. We denote by $\Omega^{1,0}$ and $\Omega^{0,1}$ the holomorphic cotangent bundle of $M$ and its complex conjugate bundle and by $\Omega^{p,q}$, with $0\le p,q \le n$, the $C^{\infty}$ complex vector bundle on $M$ obtained by taking wedge products ($p$ and $q$ times) of $\Omega^{1,0}$ and $\Omega^{0,1}$. If $1\le s \le 2n$ we have $\Omega^{s}= \oplus_{p+q = s} \Omega^{p,q}$ and we denote by $A^{s}$ and $A^{p,q}$ (resp. $A^{s}(E)$ and $A^{p,q}(E)$) the spaces of $C^{\infty}$ complex $s$-forms and $(p,q)$-forms on $M$ (resp. forms on $M$ with coefficients in $E$), i.e., they are the spaces of $C^{\infty}$ sections of $\Omega^{s}$ and $\Omega^{p,q}$ (resp. of   $\Omega^{s}\otimes E$ and $\Omega^{p,q}\otimes E$). As it is known, any hermitian metric $h$ in $E$ determines a unique connection $D_{h}$ in $E$, usually known as the {\it Chern} or the {\it hermitian connection} of $h$. Following the notation in \cite{Kobayashi} we decompose $D_{h} = D'_{h} + d''$ into its $(1,0)$ and $(0,1)$ parts. The curvature of this connection is given by $R_{h} = D_{h}\wedge D_{h}$ and is an ${\rm End}E$-valued form of type $(1,1)$, usually called the {\it Chern} or the {\it hermitian curvature} of $h$. Hence, $R_{h}$ is an element in $A^{1,1}({\rm End}E)$. Using the above decomposition in types of $D_{h}$ we have 
\begin{equation*}
R_{h} = D'_{h}\wedge d'' + d'' \wedge D'_{h}\,. 
\end{equation*}

Associated to the curvature $R_{h}$, one has the {\it Chern mean curvature} $K_{h}$, which is the element in $A^{0}({\rm End}E)$ defined by\footnote{Some sources \cite{Bruzzo-Granha,Seaman} equivalently define the mean curvature operator as $K_{h} = i\Lambda_\omega R_{h}$, where $\Lambda_{\omega}$ is the dual of the  \emph{Lefschetz operator} $L=\omega\wedge\bullet$ (see \cite[p.~99]{Kobayashi} for an explanation of this equivalence).} 
\begin{equation}
in R_{h} \wedge \omega^{n-1} = K_{h}\,\omega^{n} .  \label{Def. K}
\end{equation}

A hermitian metric $h$ is said to be {\it Hermitian-Yang-Mills} if $K_{h} = cI$, where $I$ denotes the identity endomorphism of $E$ and $c$ is a constant, which is indeed fixed by invariants of the bundle (see \cite{Kobayashi}, p. 178). If $h$ is Hermitian-Yang-Mills, the pair $(E,h)$ is called an {\it Hermitian-Yang-Mills vector bundle}.\footnote{Sometimes the terminology {\it Hermitian-Einstein metric} is also used, the latest one was indeed introduced by Kobayashi \cite{Kobayashi} as a generalization to holomorphic vector bundles of the notion of {\it K\"ahler-Einstein metric} for complex K\"ahler manifolds.}

Given an $h$ in $E$, the pair $(E,h)$ is called a {\it holomorphic hermitian vector bundle} and it is common to say that $D_{h}$ and $R_{h}$ are the Chern connection and curvature of $(E,h)$, respectively. Clearly, we can consider $C^{\infty}$ complex subbundles of $E$ and not only holomorphic subbundles. In particular, if $E'\subset E$ is a $C^{\infty}$ complex subbundle we define the orthogonal complement $E''$ of $E'$ with respect to $h$, which is also a $C^{\infty}$ subbundle of $E$. The bundles $E'$ and $E''$ are not necessarily holomorphic and hence they give a $C^{\infty}$ orthogonal decomposition $E = E' \oplus E''$. However, there exists a proposition in complex geometry that guarantees that the above decomposition is also holomorphic if a certain invariant condition holds. To be precise one has the following result.

\begin{pro} {\rm (\cite{Kobayashi}, Prop. 4.18, p. 13).} \label{Prop. 1, classical}
Let $(E,h)$ be a holomorphic hermitian vector bundle and $D_{h}$ its Chern connection. Let $E'\subset E$ be a $D_{h}$-invariant $C^{\infty}$ complex subbundle and $E''$ the orthogonal complement of $E'$ with respect to $h$. Then both $E'$ and $E''$ are $D_h$-invariant holomorphic subbundles of $E$ and they give a holomorphic orthogonal decomposition:
\begin{equation}
E= E' \oplus E''.  \label{ort. dec. classical} 
\end{equation}   
\end{pro}   

There exists an important application of Proposition \ref{Prop. 1, classical} involving the holonomy of the Chern connection of $(E,h)$ and which we briefly summarize here (for more details see \cite{Kobayashi}, p. 107). Let $\Psi(x)$ be the holonomy group of $D_{h}$ with reference point $x\in M$ and let
\begin{equation}
E_{x} = E_{x}^{0} \oplus E_{x}^{1} \oplus \cdots \oplus E_{x}^{k}   \label{ort.dec. classical, fibre}
\end{equation}
be the orthogonal decomposition of the fibre $E_{x}$ such that $\Psi(x)$ acts trivially on $E_{x}^{0}$ and irreducibly on each $E_{x}^{1},..., E_{x}^{k}$. Of course $E_{x}^{0}$ may be the trivial vector space. Since $D_{h}$ is the Chern connection, $h$ is parallel with respect to it. Therefore, by taking a parallel displacement of \eqref{ort.dec. classical, fibre} we obtain an orthogonal decomposition
\begin{equation}
E = E^{0} \oplus E^{1} \oplus \cdots \oplus E^{k}    \label{ort.dec. classical}
\end{equation}
where $E^{0}$ is trivial as a hermitian vector bundle. Since $E^{0},..., E^{k}$ are all $D_{h}$-invariant $C^{\infty}$ complex subbundles of $E$, Proposition  \ref{Prop. 1, classical} implies that \eqref{ort.dec. classical}  is also holomorphic, i.e., it is a holomorphic orthogonal decomposition of $E$.

As it is well-known, if $S\subset E$ is a rank $p$ holomorphic subbundle we have the following short exact sequence 
\begin{equation}
0 \longrightarrow S \longrightarrow E \longrightarrow Q \longrightarrow 0   \label{classical sec.}
\end{equation}
where $Q=E/S$ is a rank $r-p$ holomorphic vector bundle. Let $h$ be a hermitian metric in $E$ and $S^{\bot}$ the orthogonal complement of $S$ with respect to $h$, then $S^{\bot}$ is a $C^{\infty}$ complex vector bundle (not necessarily a holomorphic one) and hence $E = S \oplus S^{\bot}$ is --a priori-- a $C^{\infty}$ complex orthogonal decomposition. As $C^{\infty}$ complex vector bundles $Q \cong S^{\bot}$ and the restriction of $h$ to $S^{\bot}$ induces --via this isomorphism-- a hermitian metric $h_{Q}$ in $Q$. Therefore, if $h_S$ denotes the restriction of $h$ to $S$, the pairs $(S,h_{S})$ and $(Q,h_{Q})$ are both holomorphic hermitian vector bundles.
 
Given any $\xi\in A^{0}(S)$ we have a decomposition 
 \begin{equation}
 D_{h}\xi = D_{h,S}\xi + A_{h}\xi\,,  \label{Def. A}
 \end{equation}
 with $D_{h,S}\xi\in A^{1}(S)$ and $A_{h}\xi\in A^{1}(S^{\bot})$, and one has the following result.
 
 \begin{pro}{\rm (\cite{Kobayashi}, Prop. 6.4, p. 20).} \label{Prop. 2, classical}
Let $(E,h)$ be a holomorphic hermitian vector bundle and $D_{h}$ its Chern connection. Let $S\subset E$ be a holomorphic subbundle. Then, in the decomposition \eqref{Def. A}, the operator $D_{h,S}$ is the Chern connection of $(S,h_{S})$ and $A_{h}$ is a $(1,0)$-form with values in ${\rm Hom}(S,S^{\bot})$, i.e., $D_{h,S} = D_{h_{S}}$ and $A_{h}\in A^{1,0}({\rm Hom}(S,S^{\bot}))$.
 \end{pro}
 In a similar form, for any $\eta\in A^{0}(S^{\bot})$ we have 
 \begin{equation}
 D_{h}\eta = B_{h}\eta + D_{h,S^{\bot}}\eta\,, \label{Def. B}
 \end{equation}
 with $D_{h,S^{\bot}}\eta\in A^{1}(S^{\bot})$ and $B_{h}\eta\in A^{1}(S)$. Via the aforementioned isomorphism we can consider 
 \begin{equation}
 D_{h,S^{\bot}} \equiv D_{h,Q} : A^{0}(Q) \longrightarrow A^{1}(Q)   \label{Def. D_h,Q}
 \end{equation}
 and one has the following result.
 \begin{pro}{\rm (\cite{Kobayashi}, Prop. 6.6, p. 21).} \label{Prop. 3, classical}
 Let $(E,h)$, $D_{h}$ and $S$ as in Proposition \ref{Prop. 2, classical} and $Q$ the holomorphic bundle given by \eqref{classical sec.}. Then the operator $D_{h,Q}$ defined in \eqref{Def. D_h,Q} is the Chern connection of $(Q,h_{Q})$ and $B_{h}$ is a $(0,1)$-form with values in ${\rm Hom}(S^{\bot},S)$, i.e., $D_{h,Q} = D_{h_{Q}}$ and $B_{h}\in A^{0,1}({\rm Hom}(S^{\bot},S))$. Moreover, for any $\xi\in A^{0}(S)$ and $\eta\in A^{0}(S^{\bot})$ we have
 \begin{equation}
  h(A_{h}\xi,\eta) + h(\xi,B_{h}\eta) = 0\,.   \label{Adj. formula}
 \end{equation}
 \end{pro}
 
The form $A_{h}$ in Proposition \ref{Prop. 2, classical} is called the {\it second fundamental form} of $S$ in $(E,h)$ and again --via the isomorphism-- one has $A_{h}\in A^{1,0}({\rm Hom}(S,Q))$ and $B_{h}\in A^{0,1}({\rm Hom}(Q,S))$. The expression \eqref{Adj. formula} says that $A_{h}$ is the adjoint of $-B_{h}$ with respect to $h$. Now, as a consequence of the Leibniz rule, \eqref{Def. A} and \eqref{Def. B} can be naturally extended to decompositions for $s$-forms evaluated on $S$ and $S^{\bot}$, respectively.\footnote{Notice that if $\lambda\in A^{s}$, then $D_{h}(\lambda\xi) = (d\lambda)\xi + D_{h}\xi\wedge\lambda = D_{h,S}(\lambda\xi) + A_{h}\wedge(\lambda\xi)\,$.} Therefore, for any $\xi\in A^{0}(S)$ we have
\begin{eqnarray*}
R_{h}\xi &=& D_{h}(D_{h,S}\xi + A_{h}\xi)\\
             &=& D_{h,S}(D_{h,S}\xi) + A_{h}\wedge(D_{h,S}\xi) + B_{h}\wedge(A_{h}\xi) + D_{h,S^{\bot}}(A_{h}\xi)\\
             &=& (R_{h,S} + B_{h}\wedge A_{h} + D_{h}A_{h})\xi\,. 
\end{eqnarray*}
Similarly, for any $\eta\in A^{0}(S^{\bot})$ we get
\begin{equation*}
R_{h}\eta = (R_{h,S^{\bot}} + A_{h}\wedge B_{h} + D_{h}B_{h})\eta\,. 
\end{equation*}
Putting all these together we can rewrite $R_{h}$ in a matrix form as follows:
\begin{eqnarray}\label{R_h Koba}
R_h =
\begin{pmatrix}
R_{h,S}+B_h\wedge A_h & D_hB_h \\
D_hA_h & R_{h,S^\bot}+A_h\wedge B_h
\end{pmatrix},
\end{eqnarray}
which are Gauss-Codazzi type equations for holomorphic hermitian vector bundles (cf. \cite{Kobayashi}, p. 23 or also \cite{Demailly}, p. 274). Now, from Propositions \ref{Prop. 2, classical} and \ref{Prop. 3, classical} the forms $A_{h}$ and $B_{h}$ are of type $(1,0)$ and $(0,1)$ and since $R_{h}\in A^{1,1}({\rm End}E)$, the last terms in the right hand sides of $R_{h}\xi$ and $R_{h}\eta$ can be replaced by $d''A_{h}$ and $D'_{h}B_{h}$. Consequently, the matrix expression \eqref{R_h Koba} can be further simplified. Finally, one has the following result.

\begin{pro} {\rm (\cite{Kobayashi}, Prop. 6.14, p. 23).} \label{Prop. 4, classical}
 Let $(E,h)$, $D_{h}$ and $S$ as in Proposition \ref{Prop. 2, classical}. If the second fundamental form $A_{h}$ of $S$ in $(E,h)$ vanishes identically, then the orthogonal complement $S^{\bot}\subset E$ is a holomorphic subbundle and 
 \begin{equation}
 E = S\oplus S^{\bot} \label{Dec. S, S^bot, classical}
 \end{equation}
 is a holomorphic orthogonal decomposition.   
\end{pro}
Proposition \ref{Prop. 4, classical} is a straightforward consequence of Proposition \ref{Prop. 1, classical}. In fact, if $A_{h}$ vanishes identically, then \eqref{Def. A} implies that $S\subset E$ is $D_{h}$-invariant. Notice that the converse of Proposition \ref{Prop. 4, classical} is also true. In fact, if the decomposition $E=S\oplus S^{\bot}$ is holomorphic, then the pairs $(S,h_{S})$ and $(S^{\bot},h_{S^{\bot}})$ are both holomorphic hermitian vector bundles. If $D_{h_{S}}$ and $D_{h_{S^{\bot}}}$ are the Chern connections of these pairs, then $D_{h}$ is the direct sum of $D_{h_{S}}$ and $D_{h_{S^{\bot}}}$ and the result follows from \eqref{Def. A} and Proposition \ref{Prop. 2, classical}, or equivalently from \eqref{Def. B} and Proposition \ref{Prop. 3, classical}. 

On the other hand, let us denote by ${\rm Herm}^{+}E$ {\it the space of hermitian metrics} in the holomorphic vector bundle $E$. If $h \in {\rm Herm}^{+}E$ is considered as a variable and $K_{h}$ is the curvature defined in \eqref{Def. K}, Kobayashi introduces the functional $J:{\rm Herm}^{+}E \longrightarrow {\mathbb R}$ defined by 
\begin{equation}
J(h) = \frac{1}{2}\int_{M}\lvert K_{h}\lvert^{2}\omega^{n} = \frac{n!}{2}\lVert K_{h}\lVert ^{2}\,  \label{Def. J}
\end{equation} 
where $\lvert \cdot \lvert$ and $\lVert\cdot\lVert$ denote the usual pointwise and $L^{2}$-norms in $A^{0}({\rm End}E)$, respectively (see \cite{Kobayashi} or \cite{Cardona 7, Cardona 9} for details). Hence, up to a multiplicative constant, $J$ is the energy functional for the Chern mean curvature. As it is well-known, the functional \eqref{Def. J} can be considered as the Yang-Mills functional\footnote{Strictly speaking, the {\it Yang-Mills functional} $I$ is the energy functional of the Chern curvature. However, the difference $I(h) - J(h)$ does not depend on $h$ and is a topological term involving the first and second Chern classes of $E$. See \cite{Kobayashi}, p. 111 for details.} and there are some classical results associated to $J$. In particular, a hermitian metric $h$ is a minimum of $J$ if and only if it is Hermitian-Yang-Mills. More in general one has the following result. 

\begin{thm} {\rm (\cite{Kobayashi}, Thm. 3.21, p. 110).} \label{Prop. 5, classical}
Let $E$ be a holomorphic vector bundle. A hermitian metric $h$ is a critical point of the functional \eqref{Def. J} if and only if the Chern mean curvature is parallel with respect to the Chern connection defined by $h$, i.e., if and only if 
\begin{equation}
D_{h}K_{h} = 0\,.   \label{DK classical}  
\end{equation} 
\end{thm} 

The parallelism condition \eqref{DK classical}, together with the parallelism of $h$ with respect to the Chern connection, has very important consequences for the holomorphic bundle $E$ and functional $J$ (see \cite{Kobayashi}, p. 111 for details). In fact, assume that $h$ is a critical point of $J$. Then, using \eqref{DK classical} we can diagonalize $K_{h}$ and obtain a decomposition of $E$ in terms of the $C^{\infty}$ subbundles. Taking a further decomposition of it using the holonomy group of $D_{h}$, we obtain a holomorphic orthogonal decomposition as in \eqref{ort.dec. classical}, where this time $E^{0}$, $E^{1}, ... , E^{k}$ are holomorphic eigenbundles of $K_{h}$ with constants $0, c_{1}, ..., c_{k}$. In summary, from the above one has the following result.

\begin{thm} {\rm (\cite{Kobayashi}, Thm. 3.27, p. 111).} \label{Prop. 6, classical}
Let $E$ be a holomorphic vector bundle. Assume that $h$ is a critical point of the functional \eqref{Def. J} and let 
\begin{equation}
E = E^{0} \oplus E^{1} \oplus \cdots \oplus E^{k}\, \label{ort.dec. classical final}
\end{equation}
be the holomorphic orthogonal decomposition of $(E,h)$ given by \eqref{ort.dec. classical}. Let $h_{0}, h_{1}, ... , h_{k}$ be the restrictions of $h$ to $E^{0}, E^{1}, ... , E^{k}$, respectively. Then, the pairs $(E^{0},h_{0}), (E^{1},h_{1}), ... , (E^{k},h_{k})$ are all Hermitian-Yang-Mills vector bundles with constants $0, c_{1}, ... , c_{k}$.
\end{thm}
 
We would like to emphasize that propositions \ref{Prop. 2, classical} and \ref{Prop. 3, classical} have been already extended to Higgs bundles and appear --although using a local component notation-- in the form of Proposition 2.3 in \cite{Bruzzo-Granha}. However, as we will explain in the following section, these results are just special cases of two more general propositions on Higgs bundles. Also, Theorem \ref{Prop. 5, classical} has been extended to Higgs bundles in \cite{Cardona 7}. However, as far as the authors know, extensions of propositions \ref{Prop. 1, classical} and \ref{Prop. 4, classical}  and Theorem \ref{Prop. 6, classical} have not been explored yet in the literature. As previously mentioned, the main purpose of this article is to study in detail possible extensions of these results to Higgs bundles. We will address all these issues in the last two sections of the article.
 
\section{Hermitian Higgs bundles and the second fundamental form} \label{Higgs 2nd. fund.}

The aim of this section, is to revisit a couple of results on hermitian Higgs bundles previously known in literature \cite{Bruzzo-Granha}. Here, we show that such results admit generalizations that can be viewed as natural extensions of propositions \ref{Prop. 2, classical} and 
\ref{Prop. 3, classical}. We include here alternative proofs of these results without using local computations. In this section we also recall relevant information about Higgs bundles and fix the remaining part of the terminology that will be used throughout the article. In order to do this, we begin here reviewing some standard definitions on Higgs bundles (for more details the reader can see \cite{Cardona 7, Cardona 9} or the pioneering work \cite{Simpson}).

A {\it hermitian Higgs bundle} is a triple $({\mathfrak E},h) = (E,\Phi, h)$, where ${\mathfrak E}=(E,\Phi)$ is a Higgs bundle \cite{Simpson} and $(E,h)$ is a holomorphic hermitian vector bundle \cite{Kobayashi}. In particular, the Higgs field 
$\Phi\in A^{1,0}({\rm End}E)$ is holomorphic and satisfies $\Phi\wedge\Phi = 0$. Using the hermitian metric $h$ one has a hermitian conjugate ${\bar\Phi}_{h}\in A^{0,1}({\rm End} E)$ of the Higgs field and we get in the Higgs bundles setting the following connection:
\begin{equation}
{\cal D}_{h} = D_{h} + \Phi + {\bar\Phi}_{h}\,,   \label{Def. HS-conn.}
\end{equation} 
usually known as the {\it Hitchin-Simpson connection} of $({\mathfrak E},h)$ \cite{Bruzzo-Granha}. A straightforward computation shows that its curvature ${\cal R}_{h} = {\cal D}_{h}\wedge{\cal D}_{h}$ is given by
\begin{equation}
{\cal R}_{h} = R_{h} + D'_{h}\Phi + d''{\bar\Phi}_{h} + [\Phi,{\bar\Phi}_{h}]\,.   \label{Def. cal R}
\end{equation}
Here, as we mentioned before, $R_{h}$ is the Chern curvature, $D'_{h}$ and $d''$ are the $(1,0)$ and $(0,1)$ parts of $D_{h}$ and the last term is the usual graded commutator in the space of 
${\rm End}E$-valued forms \cite{Cardona 9}, i.e.,   
\begin{equation*}
 [\Phi,{\bar\Phi}_{h}] =  \Phi\wedge{\bar\Phi}_{h} + {\bar\Phi}_{h}\wedge\Phi\,. 
\end{equation*}

If ${\mathfrak E} = (E,\Phi)$ and ${\mathfrak E}' = (E',\Phi')$ are  Higgs bundles (over the same $M$), a {\it morphism of Higgs bundles} $f: {\mathfrak E}\longrightarrow {\mathfrak E'}$ is just a mapping  
$f: E\longrightarrow E'$ such that 
\begin{equation*}
\Phi'\circ f = (f\otimes I)\circ\Phi\,, 
\end{equation*}
where $I$ denotes here the identity morphism of $\Omega^{1,0}$, i.e., it is a morphism of the corresponding 
holomorphic vector bundles in which the obvious square diagram commutes. If ${\mathfrak E} = (E,\Phi)$ is a Higgs bundle, a {\it Higgs subbundle} ${\mathfrak S}\subset {\mathfrak E}$ is a $\Phi$-invariant holomorphic subbundle $S\subset E$. Hence, if $\Phi\lvert_S$ is the restriction of the Higgs field $\Phi$ to $S$, the pair $(S,\Phi\lvert_S)$ is itself a Higgs bundle and the inclusion induces a morphism of Higgs bundles ${\mathfrak S} \longrightarrow {\mathfrak E}\,$. Whenever a metric $h_S$ is defined, we simply write ${\bar\Phi}_{h_S}$ to denote the hermitian conjugate of $\Phi\lvert_S$ with respect to $h_S$. If $\mathfrak{S}$ is a Higgs subbundle and $P_S$ and $P_{S^{\perp}}$ are the standard projections onto $S$ and $S^{\perp}$, we denote by $P_S\Phi$ and $P_{S^{\perp}}\Phi$ the compositions of $P_S$ and $P_{S^{\perp}}$ with $\Phi$. A similar definition applies to ${\bar\Phi}_h$.

The following is a well-known result on Higgs bundles, however it plays a fundamental role in this article and hence we include it explicitly. 

\begin{lem} \label{Lem. 1}
If ${\mathfrak S}\subset{\mathfrak E}$ is a Higgs subbundle, there exists a short exact sequence of Higgs bundles
\begin{equation}
0 \longrightarrow {\mathfrak S} \longrightarrow {\mathfrak E} \longrightarrow {\mathfrak Q} \longrightarrow 0 \,   \label{Higgs sec.}
\end{equation}    
where ${\mathfrak Q}$ is the quotient Higgs bundle.   
\end{lem}
\begin{proof}
Associated to $S\subset E$ we have the short exact sequence \eqref{classical sec.} with $Q=E/S$ holomorphic. By tensoring this sequence with $\Omega^{1,0}$ we get a commutative diagram with an induced Higgs field $\Phi_Q\in A^{1,0}(\mathrm{End}Q)$ for $Q$. Hence the pair $\mathfrak{Q}=(Q,\Phi_Q)$ is a Higgs bundle and we get \eqref{Higgs sec.}. 
\end{proof}

Let $({\mathfrak E},h)$ be a hermitian Higgs bundle and ${\mathfrak S}\subset{\mathfrak E}$ a Higgs subbundle. Since $S$ is $\Phi$-invariant it follows that $S^{\bot}$ is ${\bar\Phi}_{h}$-invariant. In fact, if $\eta\in A^{0}(S^{\bot})$ and $\xi\in A^{0}(S)$, then we get $h(\bar\Phi_{h}\eta,\xi) = h(\eta,\Phi \xi) = 0$ and $\bar\Phi_{h}\eta \in A^{1}(S^{\bot})\,$. Moreover, in general $S^{\perp}$ is not $\Phi$-invariant\footnote{Let $M$ be an $1$-dimensional compact complex manifold of genus $1$, let $\mathcal{O}_M$ be its structure sheaf and choose a nowhere-vanishing $1$-form $\varphi$. Set $E=\mathcal{O}_M\oplus\mathcal{O}_M$ and define $\Phi\in A^{1,0}(\mathrm{End}E)$ by
\begin{align*}
\Phi = 
\begin{pmatrix}
0 & \varphi \\
0 & 0
\end{pmatrix}.
\end{align*}
The pair $\mathfrak{E} = (E,\Phi)$ is a rank $2$ Higgs bundle over $M$. Let $h$ be the product hermitian metric on $E$ and consider a holomorphic orthogonal frame $\{s_1,s_2\}$, so the subbundles $S = \langle s_1\rangle$ and $S^{\perp} = \langle s_2\rangle$ give an orthogonal decomposition $E = S\oplus S^{\perp}$ with $A_h = 0$ and such that $\Phi s_1 = 0$ and $\Phi s_2 = \varphi s_1$. Thus $S$ is $\Phi$-invariant but $S^{\perp}$ is not. Consequently, from Proposition \ref{Prop. 4, classical} the decomposition $E = S \oplus S^{\perp}$ is also holomorphic and the pair $\mathfrak{S} = (S,\Phi\lvert_S)$ is a Higgs subbundle (see \cite{Hitchin} for more details on this type of constructions).} and, as we will see in Section \ref{Herm. Higgs b.}, this fact plays a crucial role in the decomposition of a Higgs bundle in terms of Higgs subbundles. If $S^{\perp}$ is not $\Phi$-invariant, a minor modification of the previous arguments implies that also $S$ is not ${\bar\Phi}_h$-invariant.\footnote{In fact, if $\xi\in A^0(S)$ and $\eta\in A^0(S^{\perp})$ it follows that $0\neq h(\xi,\Phi\eta) = h({\bar\Phi}_h\xi,\eta)$ and hence, in general, ${\bar\Phi}_h\xi$ has a component in $A^1(S^{\perp})$. } Consequently, given any section $\xi\in A^0(S)$ we have
\begin{equation} \label{Phi bar in S}
{\bar\Phi}_h\xi = P_S{\bar\Phi}_h\xi + P_{S^{\perp}}{\bar\Phi}_h\xi.
\end{equation}
For future reference, we introduce the following terminology. If ${\mathfrak S}$ is a Higgs subbundle of ${\mathfrak E}$ and the holomorphic subbundle $S\subset E$ is $D_{h}$-invariant, we say that the Higgs subbundle $\mathfrak S$ is $D_{h}$-invariant. A similar definition applies to the Hitchin-Simpson connection ${\cal D}_{h}$.

If now $\xi\in A^{0}(S)$ and we use \eqref{Def. A}, \eqref{Def. HS-conn.} and \eqref{Phi bar in S} we get a decomposition
\begin{equation*}
{\cal D}_{h}\xi = (D_{h,S} + \Phi + P_S{\bar\Phi}_{h})\xi + (A_{h} + P_{S^{\perp}}{\bar\Phi}_{h})\xi \,, 
\end{equation*}
with $(D_{h,S} + \Phi + P_S{\bar\Phi}_{h})\xi\in A^{1}(S)$ and $(A_{h} + P_{S^{\perp}}{\bar\Phi}_{h})\xi\in A^{1}(S^{\bot})$. One observes that the $\mathrm{End}S$-valued form $P_S{\bar\Phi}_{h}\lvert_S:S\longrightarrow S\otimes\Omega^{0,1}$ is just the hermitian conjugate of $\Phi\lvert_S$ with respect to $h_{S}$.\footnote{If $\xi,\eta\in A^{0}(S)$ then 
\begin{equation*}
h_S({\bar\Phi}_{h_S}\xi,\eta) = h_S(\xi,\Phi\eta) = h(\xi,\Phi\eta) = h({\bar\Phi}_h\xi,\eta) = h(P_S{\bar\Phi}_{h}\xi,\eta) = h_S(P_S{\bar\Phi}_{h}\xi,\eta),
\end{equation*}
where in the penultimate equality we have used \eqref{Phi bar in S}. Therefore ${\bar\Phi}_{h_S} = P_S{\bar\Phi}_{h}\lvert_S$.}
By defining 
\begin{equation}
{\cal D}_{h,S} = D_{h,S} + \Phi\lvert_S + P_S{\bar\Phi}_{h}\lvert_S\, \label{HS conn. of S}
\end{equation}
and 
\begin{equation}\label{eq:A_cal}
{\cal A}_{h} = A_h + P_{S^{\perp}}{\bar\Phi}_{h}\lvert_S 
\end{equation}
we can rewrite the above decomposition as:
\begin{equation}
{\cal D}_{h}\xi = {\cal D}_{h,S}\xi + {\cal A}_{h}\xi\,. \label{Def. A Higgs}
\end{equation}

The expression \eqref{Def. A Higgs} is the natural extension to hermitian Higgs bundles of the classical decomposition for holomorphic vector bundles \eqref{Def. A}. From \eqref{HS conn. of S}, \eqref{Def. A Higgs} and Proposition 
\ref{Prop. 2, classical} we immediately have the following result.
\begin{pro} \label{Prop. 2, Higgs}
Let $({\mathfrak E},h)$ be a hermitian Higgs bundle and ${\cal D}_{h}$ its Hitchin-Simpson connection. Let ${\mathfrak S}\subset {\mathfrak E}$ be a Higgs subbundle. Then the operator ${\cal D}_{h,S}$ defined in \eqref{HS conn. of S} is the Hitchin-Simpson connection of $({\mathfrak S},h_{S})$ and ${\cal A}_{h}$ is a $1$-form with values in $\mathrm{Hom}(S,S^\perp)$, i.e., ${\cal D}_{h,S} = {\cal D}_{h_{S}}$ and ${\cal A}_{h}\in A^{1}({\rm Hom}(S,S^{\bot}))$.
\end{pro}
This result can be seen as an extension to Higgs bundles of Proposition \ref{Prop. 2, classical}. We would like to emphasize that, strictly speaking, Proposition \ref{Prop. 2, Higgs} is also a generalization of Proposition 2.3 (parts (i) and (ii)) in \cite{Bruzzo-Granha}. Moreover, their result is established and proved using a different approach. In particular, in that work the authors introduce the main objects using local unitary frame fields. Here we make explicit that Proposition \ref{Prop. 2, Higgs} is closely related to a well-known result in holomorphic hermitian vector bundles and that it can be obtained without using local computations. Since in general $S^{\perp}$ is not $\Phi$-invariant, given any section $\eta\in A^0(S^{\perp})$ we have
\begin{equation}\label{Phi in S perp}
\Phi\eta = P_{S^{\perp}}\Phi\eta + P_S\Phi\eta.
\end{equation}
On the other hand, if $\eta\in A^{0}(S^{\bot})$ and we use \eqref{Def. HS-conn.} and \eqref{Def. B} we get a decomposition
\begin{equation*}
{\cal D}_{h}\eta = (B_{h} + P_S\Phi)\eta + (D_{h,S^{\bot}} + P_{S^{\perp}}\Phi + \bar\Phi_{h})\eta \,, 
\end{equation*}
with $(D_{h,S^{\perp}} + P_{S^{\perp}}\Phi + \bar\Phi_{h})\eta\in A^{1}(S^{\bot})$ and $(B_{h} + P_S\Phi)\eta\in A^{1}(S)$. Notice that this time we are using that $S^{\bot}$ is $\bar\Phi_{h}$-invariant and the $\mathrm{End}S^{\perp}$-valued form ${\bar\Phi}_h\lvert_{S^{\perp}}:S^{\perp}\longrightarrow S^{\perp}\otimes\Omega^{0,1}$ is the hermitian conjugate of $P_{S^{\perp}}\Phi\lvert_{S^{\perp}}:S^{\perp}\longrightarrow S^{\perp}\otimes\Omega^{1,0}$ with respect to $h_{S^{\perp}}$.\footnote{In fact, if $\xi,\eta\in A^{0}(S^{\perp})$ we obtain
\begin{equation*}
h_{S^{\perp}}((\overline{P_{S^{\perp}}\Phi})_{h_{S^{\perp}}}\xi,\eta) = h_{S^{\perp}}(\xi,P_{S^{\perp}}\Phi\eta) = h(\xi,P_{S^{\perp}}\Phi\eta) = h(\xi,\Phi\eta) = h({\bar\Phi}_h\xi,\eta) = h_{S^{\perp}}({\bar\Phi}_{h}\xi,\eta),
\end{equation*}
where the third equality follows from \eqref{Phi in S perp}. Consequently $(\overline{P_{S^{\perp}}\Phi})_{h_{S^{\perp}}} = {\bar\Phi}_{h}\lvert_{S^{\perp}}$.}  By defining 
\begin{equation}
{\cal D}_{h,S^{\bot}} =  D_{h,S^{\bot}} + P_{S^{\perp}}\Phi\lvert_{S^{\perp}} + \bar\Phi_{h}\lvert_{S^{\perp}}\, \label{Op. conn. of S orth.}
\end{equation}
and
\begin{equation}
{\cal B}_h = B_{h} + P_S\Phi\lvert_{S^{\perp}},   \label{eq:B_cal} 
\end{equation}
we can rewrite the above decomposition as: 
\begin{equation}
{\cal D}_{h}\eta = {\cal B}_{h}\eta + {\cal D}_{h,S^{\bot}}\eta\,. \label{Def. B Higgs}
\end{equation}
The expression \eqref{Def. B Higgs} is the natural extension to hermitian Higgs bundles of the classical decomposition for holomorphic hermitian vector bundles \eqref{Def. B}. Now, via the aforementioned isomorphism and using \eqref{Def. D_h,Q}
we can consider 
\begin{equation}
{\cal D}_{h,S^{\bot}} \equiv {\cal D}_{h,Q} : A^{0}(Q) \longrightarrow A^{1}(Q)  \label{Def. {cal D}_h,Q}
\end{equation} 
and from \eqref{Op. conn. of S orth.}-\eqref{Def. {cal D}_h,Q}, Proposition \ref{Prop. 3, classical} and Lemma \ref{Lem. 1} we immediately get the following result.

 \begin{pro} \label{Prop. 3, Higgs}
 Let $({\mathfrak E},h)$, ${\cal D}_{h}$ and ${\mathfrak S}$ as in Proposition \ref{Prop. 2, Higgs} and ${\mathfrak Q}$ the quotient 
 Higgs bundle given by Lemma \ref{Lem. 1}. Then the operator ${\cal D}_{h,Q}$ defined in 
 \eqref{Def. {cal D}_h,Q} is the Hitchin-Simpson connection of $({\mathfrak Q},h_{Q})$ and ${\cal B}_{h}$ is a $1$-form with values in ${\rm Hom}(S^{\bot},S)$, i.e., ${\cal D}_{h,Q} = {\cal D}_{h_{Q}}$ and ${\cal B}_{h}\in A^{1}({\rm Hom}(S^{\bot},S))$.
 \end{pro}
 
The above proposition is a Higgs extension of the classical Proposition \ref{Prop. 3, classical}. Again, here we make explicit that Proposition \ref{Prop. 3, Higgs} is closely related to a well-known result in holomorphic hermitian vector bundles and that it can be obtained without using local computations. In general, the forms $\mathcal{A}_{h}$ and $\mathcal{B}_{h}$ in propositions \ref{Prop. 2, Higgs} and \ref{Prop. 3, Higgs} are not equal to the classical forms $A_{h}$ and $B_{h}$ of the corresponding holomorphic objects. In fact, the operators $P_{S^{\perp}}{\bar\Phi}_h\lvert_S$ and $P_S\Phi\lvert_{S^{\perp}}$ in \eqref{eq:A_cal} and  \eqref{eq:B_cal} represent the obstructions to these equalities. To be more precise, if $({\mathfrak E},h)$ is a hermitian Higgs bundle and ${\mathfrak S}\subset{\mathfrak E}$ is a Higgs subbundle, the forms ${\cal A}_{h}$ and ${\cal B}_{h}$ of ${\mathfrak S}$ in $({\mathfrak E},h)$ are not equal to the classical forms $A_{h}$ and $B_{h}$ of $S$ in $(E,h)$ unless additional conditions are imposed. Now, as in the case of holomorphic hermitian vector bundles, \eqref{Def. A Higgs} and \eqref{Def. B Higgs} can be extended to decompositions for $s$-forms evaluated on $S$ and $S^{\bot}$, respectively. Therefore, for any 
$\xi\in A^{0}(S)$ we get
\begin{equation*}
{\cal R}_{h}\xi = ({\cal R}_{h,S} + {\cal B}_{h}\wedge {\cal A}_{h}+ {\cal D}_{h}{\cal A}_{h})\xi\,,
\end{equation*}
and if  $\eta\in A^{0}(S^{\bot})$ we obtain
\begin{equation*}
{\cal R}_{h}\eta = ({\cal R}_{h,S^{\bot}} + {\cal A}_{h}\wedge {\cal B}_{h} + {\cal D}_{h}{\cal B}_{h})\eta\,.
\end{equation*}
As it is usual, the above expressions can be used to rewrite ${\cal R}_{h}$ in a matrix form as:
\begin{eqnarray}\label{R_h Higgs}
\mathcal{R}_h =
\begin{pmatrix}
\mathcal{R}_{h,S} + {\cal B}_h\wedge {\cal A}_h & \mathcal{D}_h{\cal B}_h \\
\mathcal{D}_h{\cal A}_h & \mathcal{R}_{h,S^\bot} + {\cal A}_h\wedge {\cal B}_h
\end{pmatrix},
\end{eqnarray}
which is the Higgs extension of the Chern curvature matrix formula \eqref{R_h Koba}. The expression \eqref{R_h Higgs} can be considered as the Gauss-Codazzi type equations for hermitian Higgs bundles.


\section{Orthogonal decompositions of hermitian Higgs bundles} \label{Herm. Higgs b.}

The aim of this section is to establish extensions of propositions \ref{Prop. 1, classical} and \ref{Prop. 4, classical} to the Higgs bundles setting. In particular, we show that these extended results on Higgs bundles can be obtained from the usual definitions and results on (holomorphic) Higgs bundles if additional invariance conditions hold. Hence, strictly speaking, the above extensions rely only on the standard notions given in the previous sections.

Let $\mathcal{S}\subset\mathfrak{E}$ be a Higgs subbundle. In order to extend propositions \ref{Prop. 1, classical} and \ref{Prop. 4, classical} lets asume that $\mathcal{S}$ is also ${\bar\Phi}_h$-invariant. Notice that this condition is equivalent to $S^{\perp}$ being $\Phi$-invariant. From the above we immediately get the following result.
\begin{lem}\label{Lem. 2}
If $\mathfrak{S} \subset \mathfrak{E}$ is a Higgs subbundle that is ${\bar\Phi}_h$-invariant, then the holomorphic subbundle $S \subset E$ is $D_h$-invariant if and only if it is $\mathcal{D}_h$-invariant.
\end{lem}

According to the terminology introduced in sections \ref{Intro.} and \ref{Higgs 2nd. fund.}, we are now in a position to establish the following result. 

\begin{pro}\label{Prop. 1, Higgs}
Let $(\mathfrak E,h)=(E,\Phi,h)$ be a hermitian Higgs bundle and $D_{h}$, $E'$ and $E''$ as in Proposition \ref{Prop. 1, classical} and suppose that $E'$ is $\Phi$-invariant and also ${\bar\Phi}_h$-invariant. Then both ${\mathfrak E}' = (E',\Phi\lvert_{E'})$ and  ${\mathfrak E}'' = (E'',\Phi\lvert_{E''})$ are $D_h$-invariant Higgs subbundles of ${\mathfrak E}=(E,\Phi)$ and they give a Higgs orthogonal decomposition: 
\begin{equation}
{\mathfrak E}= {\mathfrak E}' \oplus {\mathfrak E}''.  \label{ort. dec. Higgs} 
\end{equation}  
\end{pro}

\begin{proof}
From Proposition \ref{Prop. 1, classical} we know that $E',E''\subset E$ are $D_h$-invariant holomorphic subbundles and we have the holomorphic orthogonal decomposition \eqref{ort. dec. classical}, i.e., 
$E= E' \oplus E''$ in the holomorphic category. Since $E'$ is holomorphic and also $\Phi$-invariant, the pair ${\mathfrak E}' = (E',\Phi\lvert_{E'})$ is a $D_{h}$-invariant Higgs subbundle of ${\mathfrak E} = (E,\Phi)$. In particular, by defining ${\mathfrak S} = {\mathfrak E}'$ it follows that ${\mathfrak S}\subset{\mathfrak E}$ is a Higgs subbundle and hence $S^{\bot}=E''$ is also $\Phi$-invariant. Since $E''$ is holomorphic, it follows that ${\mathfrak E}'' = (E'',\Phi\lvert_{E''})$ is a $D_{h}$-invariant Higgs subbundle of ${\mathfrak E} = (E,\Phi)$ and we obtain  \eqref{ort. dec. Higgs}.
\end{proof}

The above result can be seen as an extension to Higgs bundles of Proposition \ref{Prop. 1, classical}. More precisely, Proposition \ref{Prop. 1, Higgs} guarantees that the same classical hypothesis, together with additional $\Phi$ and ${\bar\Phi}_h$-invariance conditions, implies that ${\mathfrak E}',{\mathfrak E}''\subset {\mathfrak E}$ are Higgs subbundles and that \eqref{ort. dec. Higgs} is a decomposition of the original Higgs bundle in terms of these Higgs subbundles. In particular, let $(\mathfrak E,h)=(E,\Phi,h)$ be a hermitian Higgs bundle and let us consider again the decomposition \eqref{ort.dec. classical} for the holomorphic bundle $E$. If $E^{0}, ..., E^{k}$ are all $\Phi$ and ${\bar\Phi}_h$-invariant and $\Phi_{0}, \Phi_{1}, ... , \Phi_{k}$ denote the restrictions of $\Phi$ to these bundles, respectively, Proposition \ref{Prop. 1, Higgs} immediately implies that ${\mathfrak E}^{0}= (E^{0},\Phi_{0}), {\mathfrak E}^{1}= (E^{1},\Phi_{1}), ... ,{\mathfrak E}^{k}= (E^{k},\Phi_{k})$ are all Higgs subbundles of the original Higgs bundle ${\mathfrak E}$. To be precise, in such a case 
\begin{equation}
{\mathfrak E} = {\mathfrak E}^{0} \oplus {\mathfrak E}^{1} \oplus \cdots \oplus {\mathfrak E}^{k}\,.  
\end{equation}
is a Higgs orthogonal decomposition of ${\mathfrak E}$.

\begin{pro}\label{Prop. 4, Higgs}
Let $(\mathfrak E,h)=(E,\Phi,h)$ be a hermitian Higgs bundle, $D_{h}$ the Chern connection of $(E,h)$ and $S\subset E$ a holomorphic subbundle. If the second fundamental form $A_{h}$ of $S$ on $(E,h)$ vanishes identically and $S$ is $\Phi$ and ${\bar\Phi}_h$-invariant, then ${\mathfrak S} = (S,\Phi\lvert_{S})$ and ${\mathfrak S}^{\bot} = (S^{\bot},\Phi\lvert_{S^{\bot}})$ are Higgs subbundles of ${\mathfrak E}=(E,\Phi)$ and 
\begin{equation}
{\mathfrak E} = {\mathfrak S} \oplus {\mathfrak S}^{\bot} \label{Dec. S, S^bot, Higgs}
\end{equation}
is a Higgs orthogonal decomposition. 
\end{pro}

\begin{proof}
From Proposition \ref{Prop. 4, classical} we know that $S^{\bot}\subset E$ is a holomorphic subbundle and we have the holomorphic orthogonal decomposition \eqref{Dec. S, S^bot, classical}, i.e., $E = S \oplus S^{\bot}$ in the holomorphic category. Since $S\subset E$ is holomorphic and $\Phi$-invariant, the pair ${\mathfrak S} = (S,\Phi\lvert_{S})$ is a Higgs subbundle of 
${\mathfrak E} = (E,\Phi)$. In particular, $S^{\bot}$ is $\Phi$-invariant and since it is also holomorphic, the pair 
${\mathfrak S}^{\bot} = (S^{\bot},\Phi\lvert_{S^{\bot}})$ is a Higgs subbundle of ${\mathfrak E} = (E,\Phi)$ and we obtain \eqref{Dec. S, S^bot, Higgs}. 
\end{proof}

The above proposition can be seen as a Higgs extension of Proposition \ref{Prop. 4, classical}. In particular, Proposition \ref{Prop. 4, Higgs} guarantees that the same classical hypothesis, together with additional $\Phi$ and ${\bar\Phi}_h$-invariance conditions, implies that ${\mathfrak S},{\mathfrak S}^{\bot}\subset {\mathfrak E}$ are Higgs subbundles and that \eqref{Dec. S, S^bot, Higgs} is a decomposition of the original Higgs bundle in terms of these Higgs subbundles.
 As in the classical case, the converse of Proposition \ref{Prop. 4, Higgs} is also true. In fact, if ${\mathfrak E}={\mathfrak S}\oplus {\mathfrak S}^{\bot}$ is a Higgs decomposition, the pairs $({\mathfrak S},h_{S})$ and $({\mathfrak S}^{\bot},h_{S^{\bot}})$ are both hermitian Higgs bundles. If ${\cal D}_{h_{S}}$ and ${\cal D}_{h_{S^{\bot}}}$ are the Hitchin-Simpson connections of these pairs, then ${\cal D}_{h}$ is the direct sum of ${\cal D}_{h_{S}}$ and ${\cal D}_{h_{S^{\bot}}}$ and the result follows from \eqref{eq:A_cal}, \eqref{Def. A Higgs} and Proposition \ref{Prop. 2, Higgs}. 

We are now in a position to mention some applications of the above results in the Higgs bundles setting. Recall that in Section \ref{Higgs 2nd. fund.} we have mentioned that the form ${\cal A}_{h}$ is the same as the second fundamental form $A_{h}$ of the corresponding holomorphic subbundle, provided the Higgs bundle satisfies an additional ${\bar\Phi}_h$-invariance condition. From this fact, the previous paragraph and Proposition \ref{Prop. 4, Higgs}, we immediately get the following result.

\begin{cor}
Let $(\mathfrak E,h)=(E,\Phi,h)$ be a hermitian Higgs bundle and ${\cal D}_{h}$ its Hitchin-Simpson connection. Let ${\mathfrak S}\subset {\mathfrak E}$ be a Higgs subbundle that is ${\bar\Phi}_h$-invariant. Then, the second fundamental form $\mathcal{A}_{h}$ of ${\mathfrak S}$ in $({\mathfrak E},h)$ given by \eqref{Def. A Higgs} vanishes identically if and only if ${\mathfrak E} = {\mathfrak S} \oplus {\mathfrak S}^{\bot}$ is a Higgs orthogonal decomposition. 
\end{cor}

We would like to emphasize that this result can be completely established in terms of objects in the Higgs category. Therefore, it is natural to wonder if there exists another proposition, closely related to Proposition \ref{Prop. 1, Higgs}, and which can be established using only objects in the Higgs category. As we will see in a moment, it is indeed the case and Lemma \ref{Lem. 2} is a key result in the proof of such a proposition. Before doing that, let us introduce the following definition. Let 
${\mathfrak E}'\subset{\mathfrak E}$ be a Higgs subbundle and 
$E''\subset E$ the orthogonal complement of $E'$ with respect to $h$. Let us consider the pair 
${\mathfrak E}'' = (E'',\Phi\lvert_{E''})$, which we call the $C^{\infty}$ {\it Higgs orthogonal complement} of the Higgs subbundle ${\mathfrak E}'$ with respect to $h$. The bundle $E''\subset E$ is just a $C^{\infty}$ complex subbundle (not necessarily holomorphic) and hence, strictly speaking, ${\mathfrak E}''$ is not necessarily a Higgs subbundle of ${\mathfrak E}$. However, that will be the case if a certain invariant condition holds. To be precise, we have the following result.   

\begin{cor} \label{Cor. 2 Higgs}
Let $(\mathfrak E,h)=(E,\Phi,h)$ be a hermitian Higgs bundle and ${\cal D}_{h}$ its Hitchin-Simpson connection. Let ${\mathfrak E}'\subset {\mathfrak E}$ be a ${\cal D}_{h}$ and ${\bar\Phi}_h$-invariant Higgs subbundle and ${\mathfrak E}''$ the $C^{\infty}$ Higgs orthogonal complement of ${\mathfrak E}'$ with respect to $h$. Then ${\mathfrak E}''\subset {\mathfrak E}$ is also a ${\cal D}_{h}$-invariant Higgs subbundle and ${\mathfrak E} = {\mathfrak E}' \oplus {\mathfrak E}''$ is a Higgs orthogonal decomposition. 
\end{cor}

\begin{proof}
By applying Lemma \ref{Lem. 2} to the Higgs subbundle ${\mathfrak E}'\subset {\mathfrak E}$ we know that $E'$ is $D_{h}$-invariant. Since it is also $\Phi$-invariant, Proposition \ref{Prop. 1, Higgs} implies the Higgs orthogonal decomposition  ${\mathfrak E} =  {\mathfrak E}' \oplus {\mathfrak E}''$ with ${\mathfrak E}''\subset {\mathfrak E}$ also a $D_{h}$-invariant Higgs subbundle. The ${\cal D}_{h}$-invariance of this last subbundle follows again from Lemma \ref{Lem. 2}.
\end{proof}


\section{The Kobayashi functional for Higgs bundles} \label{Sec. cal J}

Let ${\mathfrak E} = (E,\Phi)$ be a Higgs bundle and $h$ a hermitian metric in ${\mathfrak E}$. If ${\cal R}_{h}$ is the Hitchin-Simpson curvature defined in \eqref{Def. cal R}, we can define the {\it Hitchin-Simpson mean curvature} ${\cal K}_{h}$ as the element in $A^{0}({\rm End}E)$ satisfying 
\begin{equation}
in {\cal R}_{h} \wedge \omega^{n-1} = {\cal K}_{h}\,\omega^{n}\,.  \label{Def. cal K}
\end{equation} 

The above definition is the natural extension to Higgs bundles of the classical definition \eqref{Def. K} for holomorphic vector bundles. A hermitian metric $h$ in ${\mathfrak E}$ is called a {\it Hermitian-Yang-Mills metric} if ${\cal K}_{h} = cI$, where $c$ is the same constant introduced in Section \ref{Intro.}. If $h$ is Hermitian-Yang-Mills, the pair $({\mathfrak E},h)$ is called a {\it Hermitian-Yang-Mills Higgs bundle}.

Let us denote by ${\rm Herm}^{+}{\mathfrak E}$ the space of hermitian metrics\footnote{It is clearly the same space ${\rm Herm}^{+}E$ defined in Section \ref{Intro.}, the change of notation is just to emphasize that we  are now interested in the Higgs bundles setting.} in the Higgs bundle ${\mathfrak E}$. If $h \in {\rm Herm}^{+}{\mathfrak E}$ is considered as a variable and ${\cal K}_{h}$ is the Hitchin-Simpson curvature defined in \eqref{Def. cal K}, then we can consider the functional ${\cal J}:{\rm Herm}^{+}{\mathfrak E} \longrightarrow {\mathbb R}$ defined by
\begin{equation}
{\cal J}(h) = \frac{1}{2}\int_{M}\lvert {\cal K}_{h}\lvert^{2}\omega^{n} = \frac{n!}{2}\lVert {\cal K}_{h}\lVert ^{2}\,  \label{Def.  cal J}
\end{equation} 
which is, up to a multiplicative constant, the natural energy functional for the Hitchin-Simpson mean curvature. The functional \eqref{Def. cal J} is called the {\it Kobayashi functional} for the Higgs bundle ${\mathfrak E}$ 
and is the obvious extension to Higgs bundles of the classical functional \eqref{Def. J} for the corresponding holomorphic vector bundle $E$. It can be shown (see \cite{Cardona 7} for details) that a hermitian metric $h$ in ${\mathfrak E}$ is Hermitian-Yang-Mills if and only if it is  a minimum of \eqref{Def. cal J} and more in general, one has the following result.

\begin{thm} {\rm (\cite{Cardona 7}, Thm. 4).} \label{Prop. 9, Higgs}
Let ${\mathfrak E}$ be a Higgs bundle. A hermitian metric $h$ is a critical point of the functional \eqref{Def.  cal J} if and only if the Hitchin-Simpson mean curvature is parallel with respect to the Hitchin-Simpson connection defined by 
$h$, i.e., if and only if 
\begin{equation}
{\cal D}_{h}{\cal K}_{h} = 0\,.   \label{DK Higgs}  
\end{equation} 
\end{thm} 
Notice that Theorem \ref{Prop. 9, Higgs} is the generalization to Higgs bundles of Theorem \ref{Prop. 5, classical} for holomorphic vector bundles. On the other hand, the Hitchin-Simpson connection ${\cal D}_{h}$ is neither compatible with the holomorphic nor the hermitian structure of $E$, a fact also mentioned in \cite{Bruzzo-Granha, Seaman}. In fact, the (0,1) part of ${\cal D}_{h}$ is $d'' + \bar\Phi_{h}$ which coincides with $d''$ if and only if $\Phi$ is identically zero. Moreover, just from the definition of the Hitchin-Simpson connection one has 
\begin{equation}
{\cal D}_{h}h = (D_{h} + \Phi + \bar\Phi_{h})h = (\Phi + \bar\Phi_{h})h\,.  \label{non par. cond.}
\end{equation}
Here we have used the fact that $h$ is parallel with respect to the Chern connection $D_{h}$. We should notice that the right hand side in the above expression does not vanish\footnote{Let $s=(s_i)$ be a local frame for $E$ and $s^*=(s^j)$ its dual frame. Then $h = \sum h_{j{\bar k}} s^{j}\otimes {\bar s}^k$ and for any connection $D$ in $E$ one has $Dh = \sum(dh_{i{\bar\jmath}} - h_{k{\bar \jmath}}\,\theta^k_i - h_{i{\bar k}}\,{\bar\theta}^k_j) s^i\otimes{\bar s}^j$ where $\theta = (\theta^k_i)$ is the connection form of $D$ (see \cite{Kobayashi}, p. 18 for details). Since $D_{h}h=0$ one has 
\begin{equation}
{\cal D}_{h}h = (\Phi + \bar\Phi_{h})h =  -\sum( h_{k{\bar \jmath}}\,\phi^k_i + h_{i{\bar k}}\,{\bar\phi}^k_j +h_{k{\bar \jmath}}(\phi^*_h)^k_i + h_{i{\bar k}}{\overline{(\phi^*_h)}}^k_j) s^i\otimes{\bar s}^j , 
\end{equation}
where $\phi^k_i$ and $({\phi^{*}_h})_{i}^k$ are the components of $\Phi$ and $\bar\Phi_{h}$. Now, ${\cal D}_{h}h=0$ if and only if its $(1,0)$ and $(0,1)$ parts are both equal to zero. In particular, if $s$ is unitary the $(1,0)$ part is $-\sum h_{k{\bar \jmath}}\,\phi^k_i - \sum h_{i{\bar k}}{\overline{(\phi^{*}_h)}}^k_j = -\phi^j_i - {\overline{(\phi^{*}_h)}}^i_j = -2\phi^j_i$ which does not vanish unless $\Phi=0$.} and hence \eqref{non par. cond.} establishes a crucial fact. Namely, for any proper Higgs bundle (i.e., one in which $\Phi$ is not identically zero), $h$ is not parallel with respect to the Hitchin-Simpson connection ${\cal D}_{h}$. As a consequence of this, it is not possible to obtain a straightforward adaptation to Higgs bundles of the classical arguments of Section \ref{Intro.} where, as we have shown before, the parallelism condition of $h$ plays a key role. In fact, the lack of a parallelism condition of the metric $h$ with respect to ${\cal D}_{h}$ shows a significant difference between the Higgs bundles setting and the classical case of holomorphic vector bundles. It is just because of the non-vanishing of \eqref{non par. cond.} that it is not possible to generalize to Higgs bundles the classical Theorem \ref{Prop. 6, classical}, which gives a classification of the critical points of the Kobayashi functional. This makes it explicit, once again, that the functional ${\cal J}$ defined in \eqref{Def. cal J} does not always satisfy the same properties of the classical counterpart $J$, introduced in \eqref{Def. J}. Indeed, similar contrasts between the Higgs bundles setting and the classical one have already been mentioned in \cite{Cardona 7}, where it was shown that the difference between ${\cal J}$ and the Yang-Mills-Higgs functional is given by a couple of terms, one of these a topological constant involving the first and second Chern classes of ${\mathfrak E}$, and the other one a term depending on the hermitian metric. A situation in clear contrast with the classical case, where the Kobayashi and the Yang-Mills functional differ only by the topological constant.

\section{Discussion and conclusions}

In general terms, the central object of this article was the notion of Higgs bundle. As previously mentioned, Higgs bundles arose as geometric objects associated with the two-dimensional reduction of the self-dual Yang–Mills equations in four dimensions. Consequently, all properties of Higgs bundles are potentially relevant not only in complex geometry, but also in Yang-Mills theory, and more in general in quantum field theory and string theory. In order to make this article accessible to a broader audience, including both geometers and mathematical physicists, we have included a Section \ref{Intro.} in the form of preliminaries, where we summarized the main classical material from complex geometry that will be important in the rest of the article. In this section, we chose to follow the notation of the classical text of Kobayashi \cite{Kobayashi}, as it provided an elegant and aesthetically pleasing framework that is also well suited to Higgs bundles. The aim of this section was to bring together several properties that appeared in different chapters of \cite{Kobayashi}, but which are often scattered throughout other classical texts on complex geometry, such as those of Demailly \cite{Demailly} and Griffiths and Harris \cite{Griffiths-Harris}. Additionally, we firmly believe that this indeed makes the article more self-contained. Also, in this section, we presented the key ideas without relying on local computations, and, as we saw, this strategy proved useful throughout the rest of the paper. It is worth noting that a similar treatment—avoiding local computations—was also employed in \cite{Ramanan}, although in the context of the of Spin bundles.

As we mentioned in the first two sections, the main purpose of this article has been to study certain geometric properties of Higgs bundles. In particular, we focused on results that can be seen as natural extensions of classical results in complex geometry to the context of Higgs bundles. In Section \ref{Higgs 2nd. fund.} we have shown that if $({\mathfrak E},h) = (E,\Phi,h)$ is a hermitian Higgs bundle and ${\mathfrak S}\subset{\mathfrak E}$ is a Higgs subbundle, then $S^{\perp}$ is ${\bar\Phi}_{h}$-invariant, but not necessarily $\Phi$-invariant. 
This fact played a crucial role and was the origin of the forms ${\cal A}_{h}$ and ${\cal B}_{h}$ that appeared in the Gauss-Codazzi equations for hermitian Higgs bundles. Here, it is important to note that, even though the form of the Gauss–Codazzi equations \eqref{R_h Higgs} was formally similar to the classical ones, the new forms ${\cal A}_{h}$ and ${\cal B}_{h}$ were not the same as $A_{h}$ and $B_{h}$. A fact that followed from the definitions \eqref{eq:A_cal} and \eqref{eq:B_cal}. In Section \ref{Herm. Higgs b.} we have shown that using additional invariance conditions we obtain some propositions in the Higgs bundle settings that can be considered as natural extensions of classical results on holomorphic hermitian vector bundles. In particular, at the end of this section we obtain a couple of corollaries that can be established only in terms of objects in the Higgs category.

Finally, as a straightforward consequence of Theorem \ref{Prop. 9, Higgs} of Section \ref{Sec. cal J}, if $h$ is a critical point of the Kobayashi functional ${\cal J}$, then the curvature ${\cal K}_{h}$ satisfies a parallelism condition with respect to the Hitchin-Simpson connection 
${\cal D}_{h}$. However, the metric $h$ is not parallel with respect to ${\cal D}_{h}$, a fact that is summarized in \eqref{non par. cond.}.  As a consequence of this, we have shown that, as far as the Kobayashi functional is concerned, a classical decomposition theorem cannot be naturally extended to Higgs bundles. We do not know yet if there exists a kind of decomposition result concerning the critical points of the functional \eqref{Def. cal J}. We hope to address this question in a subsequent article.\\

\noindent{\bf Acknowledgements}\\

\noindent The first author wants to thank H. Garc\'ia-Compe\'an for bringing Wijnholt's paper \cite{Wijnholt} for their attention. The second author wants to thank the support of Secretar\'ia de Ciencia, Humanidades,  Tecnolog\'ia e Innovaci\'on (Secihti) through the doctoral scholarship No. 830628.


\end{document}